\documentclass[12pt]{amsart}
\usepackage{amssymb,upref,enumerate}

\usepackage{amsmath,amssymb,amscd,amsthm,amsxtra}
\usepackage{latexsym}
\usepackage{graphics,epsfig}
\allowdisplaybreaks

\def\XXint#1#2#3{{\setbox0=\hbox{$#1{#2#3}{\int}$}
\vcenter{\hbox{$#2#3$}}\kern-.5\wd0}}

\newtheorem{theorem}{Theorem}[section]
\newtheorem{proposition}{Proposition}[section]

\newtheorem{lemma}{Lemma}

\newtheorem{corollary}{Corollary}

\theoremstyle{remark}

\def\({\left(}
\def\){\right)}
\def\be {\begin{equation}}
\def\en{\end{equation}}
\def\lone{\mathcal{L}_1}
\def\ltwo{\mathcal{L}_2}
\def\babsl{\left \lvert}
\def\babsr{\right \rvert}

\def\Cdot{{\dot{C}}}
\def\Li{{\mathcal{L}_1}}

\def\div{\text{div}~}

\newcommand{\tensor}{\otimes}

\numberwithin{equation}{section}

\begin{document}



\keywords{Navier-Stokes, global existence}

\address{Nathan Pennington, Department of Mathematics, Creighton University, 319 Eppley Hall
Omaha, NE, USA.} \email{nathanpennington@creighton.edu}


\title[Generalized Leray-alpha equations]{Local and Global low-regularity solutions to the Generalized Leray-alpha equations}
\author{Nathan Pennington}
\keywords{Leray-alpha model, Besov space, fractional Laplacian}
\date{\today}

\maketitle

\bigskip

\begin{abstract}It has recently become common to study many different approximating equations of the Navier-Stokes equation.  One of these is the Leray-$\alpha$ equation, which regularizes the Navier-Stokes equation by replacing (in most locations) the solution $u$ in the equation with $(1-\alpha^2\triangle)u$ the operator $(1-\alpha^2\triangle)$.  Another is the generalized Navier-Stokes equation, which replaces the Laplacian with a Fourier multiplier with symbol of the form $|\xi|^\gamma$ ($\gamma=2$ is the standard Navier-Stokes equation), and recently in \cite{taolog} Tao also considered multipliers of the form $|\xi|^\gamma/g(|\xi|)$, where $g$ is (essentially) a logarithm.  The generalized Leray-$\alpha$ equation combines these two modifications by incorporating the regularizing term and replacing the Laplacians with more general Fourier multipliers, including allowing for $g$ terms similar to those used in \cite{taolog}.  Our goal in this paper is to obtain existence and uniqueness results with low regularity and/or non-$L^2$ initial data.  We will also use energy estimates to extend some of these local existence results to global existence results.

\end{abstract}

\section{Introduction}

The incompressible form of the Navier-Stokes equation is given by
\begin{equation}\label{NS}\aligned \partial_t u + (u\cdot \nabla)u=\nu\triangle u-\nabla p,
\\ u(0,x)=u_0(x), ~~\div (u)=0
\endaligned
\end{equation}
where $u:I\times \mathbb{R}^n\rightarrow \mathbb{R}^n$ for some time strip $I=[0,T)$, $\nu>0$ is a constant due to the viscosity of the fluid, $p:I\times\mathbb{R}^n\rightarrow \mathbb{R}^n$ denotes the fluid pressure, and $u_0:\mathbb{R}^n\rightarrow\mathbb{R}^n$.  The requisite differential operators are defined by $\triangle=\sum_{i=1}^n \frac{\partial^2}{\partial_{x_i}^2}$ and $\nabla=\(\frac{\partial} {\partial_{x_i}},...,\frac{\partial}{\partial_{x_n}}\)$.  

In dimension $n=2$, local and global existence of solutions to the Navier-Stokes equation are well known (see \cite{lady356}; for a more modern reference, see Chapter $17$ of \cite{T3}).  For dimension $n\geq 3$, the problem is significantly more complicated.  There is a robust collection of local existence results, including \cite{Kato}, in which Kato proves the existence of local solutions to the Navier-Stokes equation with initial data in $L^n(\mathbb{R}^n)$; \cite{KP}, where Kato and Ponce solve the equation with initial data in the Sobolev space  $H^{n/p-1,p}(\mathbb{R}^n)$; and \cite{KT}, where Koch and Tataru establish local existence with initial data in the space $BMO^{-1}(\mathbb{R}^n)$ (for a more complete accounting of local existence theory for the Navier-Stokes equation, see \cite{nsbook}).  In all of these local results, if the initial datum is assumed to be sufficiently small, then the local solution can be extended to a global solution.   However, the issue of global existence of solutions to the Navier-Stokes equation in dimension $n\geq 3$ for arbitrary initial data is one of the most challenging open problems remaining in analysis.

Because of the intractability of the Navier-Stokes equation, many approximating equations have been studied.  One of these is the Leray-$\alpha$ model, which is 
\begin{equation*}\aligned \label{leray1}\partial_t(1-\alpha^2\triangle )u+\nabla_u (1-\alpha^2 \triangle) u-\nu\triangle (1-\alpha^2 \triangle )u=-\nabla p,   
\\ u(0,x)=u_0(x), ~~ \div u_0=\div u=0.
\endaligned
\end{equation*}
Note that setting $\alpha=0$ returns the standard Navier-Stokes equation.  Like the Lagrangian Averaged Navier-Stokes equation (which differs from the Leray-$\alpha$ in the presence of an additional nonlinear term), the system $(\ref{leray1})$ compares favorably with numerical data; see \cite{holmleray}, in which the authors compared the Reynolds numbers for the Leray-$\alpha$ equation and the LANS equation with the Navier-Stokes equation.

Another commonly studied equation is the generalized Navier-Stokes equation, given by 
\begin{equation*}\label{NS}\aligned \partial_t u + (u\cdot \nabla)u=\nu\mathcal{L} u-\nabla p,
\\ u(0,x)=u_0(x), ~~\div (u)=0
\endaligned
\end{equation*}
where $\mathcal{L}$ is a Fourier multiplier with symbol $m(\xi)=-|\xi|^\gamma$ for $\gamma>0$.  Choosing $\gamma=2$ returns the standard Navier-Stokes equation.  In \cite{Wu}, Wu proved (among other results) the existence of unique local solutions for this equation provided the data is in the Besov space $B^s_{p,q}(\mathbb{R}^n)$ with $s=1+n/p-\gamma$ and $1<\gamma\leq 2$.   If the norm of the initial data is sufficiently small, these local solutions can be extended to global solutions.

It is well known that if $\gamma\geq \frac{n+2}{2}$, then this equation has a unique global solution.  In \cite{taolog}, Tao strengthened this result, proving global existence with the symbol $m(\xi)=-|\xi|^\gamma/g(|\xi|)$, with $\gamma\geq \frac{n+2}{2}$ and $g$ a non-decreasing, positive function that satisfies 
\begin{equation*}\int_1^\infty \frac{ds}{sg_1(s)^2}=+\infty.
\end{equation*}
Note that $g(|x|)=\log^{1/2}(2+|x|^2)$ satisfies the condition.  Similar types of results involving $g$ terms that are, essentially, logs have been proven for the nonlinear wave equation; see \cite{taolog} for a more detailed description.

Here we consider a combination of these two models, called the generalized Leray-$\alpha$ equation, which is 
\begin{equation}\aligned \label{leray1}\partial_t(1-\alpha^2 \ltwo )u+\nabla_u (1-\alpha^2  \ltwo) u-\nu\lone (1-\alpha^2 \ltwo )u=-\nabla p,   
\\ u(0,x)=u_0(x), ~~ \div u_0=\div u=0,
\endaligned
\end{equation}
with the operators $\mathcal{L}_i$ defined by 
\begin{equation*}\mathcal{L}_iu(x)=\int -\frac{|\xi|^{\gamma_i}}{g_i(\xi)}\hat{u}(\xi)e^{ix\cdot \xi}d\xi,
\end{equation*}
where $g_i$ are radially symmetric, nondecreasing, and bounded below by $1$.  Note that choosing $g_1=g_2=1$, $\gamma_1=2$, and $\gamma_2=0$ returns the Navier-Stokes equation, choosing $g_1=g_2=1$ and $\gamma_1=\gamma_2=2$ gives the Leray-$\alpha$ equation, and choosing $g_2=\gamma_2=1$ returns the generalized Navier-Stokes equation.   

In \cite{kazuo}, Kazuo obtains a unique global solution to equation $(\ref{leray1})$ in dimension three provided the initial data is in the Sobolev space $H^{m,2}(\mathbb{R}^3)$, with $m>\max\{5/2, 1+2\gamma_1\}$, provided $\gamma_1$ and $\gamma_2$ satisfy the inequality $2\gamma_1+\gamma_2\geq 5$ and that
\begin{equation}\label{g restriction} \int_1^\infty \frac{ds}{s g_1^2(s)g_2(s)} =\infty.
\end{equation}

The goal of this paper is to obtain a much wider array of existence results, specifically existence results for initial data with low regularity and for initial data outside the $L^2$ setting.  We will also, when applicable, use the energy bound from \cite{kazuo} to extend these local solutions to global solutions.  Our plan is to follow the general contraction-mapping based procedure outlined by Kato and Ponce in \cite{KP} for the Navier-Stokes equation, with two key modification.  

First, the approach used in \cite{KP} relies heavily on operator estimates for the heat kernel $e^{t\triangle}$.  We will require similar estimates for our solution operator $e^{t\mathcal{L}_1}$, and establishing these estimates is the topic of Section $\ref{operator est}$.   This will require some technical restrictions on the choices of $g_1$ and $g_2$ that will be more fully addressed below.  We also note that these estimates are the main theoretical obstacle to applying these techniques to other equations, like the generalized MHD equation studied in \cite{generalizedmhd} and the Navier-Stokes like equation studied in \cite{olsontitialphalike}.

The second modification is in how we will deal with the nonlinear term.  For the first set of results, we will use the standard Leibnitz-rule estimate to handle the nonlinear terms.  Our second set of results rely on a product estimate (due to Chemin in \cite{chemin}) which will allow us to obtain lower regularity existence but will (among other costs) require us to work in Besov spaces.  The advantages and disadvantages of each approach will be detailed later in this introduction.  The product estimates themselves are stated as Proposition $\ref{product est 1}$ and Proposition $\ref{product est 2}$ in Section $\ref{Besov space}$ below.

The rest of this paper is organized as follows.  The remainder of this introduction is devoted to stating and contextualizing the main results of the paper.  Section $\ref{Besov space}$ reviews the basic construction of Besov spaces and states some foundational results, including our two product estimates.  In Section $\ref{type 1}$ we carry out the existence argument using the standard product estimate, and in Section $\ref{type 2}$ we obtain existence results using the other product estimate.  As stated above, Section $\ref{operator est}$ contains the proofs of the operator estimates that are central to the arguments used in Sections $\ref{type 1}$ and $\ref{type 2}$.  

Our last task before stating the main results is to establish some notation.  First, we denote Besov spaces by $B^s_{p,q}(\mathbb{R}^n)$, with norm denoted by $\|\cdot\|_{B^s_{p,q}}=\|\cdot\|_{s,p,q}$ (a complete definition of these spaces can be found in Section $\ref{Besov space}$).  We define the space
\begin{equation*}C^T_{a;s,p,q}=\{f\in
C((0,T):B^s_{p,q}(\mathbb{R}^n)):\|f\|_{a;s,p,q}<\infty\},
\end{equation*}
where
\begin{equation*}\|f\|_{a;s,p,q}=\sup\{t^a\|f(t)\|_{s,p,q}:t\in
(0,T)\},
\end{equation*}
$T>0$, $a\geq 0$, and $C(A:B)$ is the space of continuous functions from $A$ to $B$.  We let $\Cdot^T_{a;s,p,q}$ denote the subspace of $C^T_{a;s,p,q}$ consisting
of $f$ such that
\begin{equation*}\lim_{t\rightarrow 0^+}t^a f(t)=0
~\text{(in}~B^s_{p,q}(\mathbb{R}^n)).
\end{equation*}
Note that while the norm $\|\cdot\|_{a;s,p,q}$ lacks an explicit reference to $T$, there is an implicit $T$ dependence. We also say $u\in BC(A:B)$ if $u\in C(A:B)$ and $\sup_{a\in A}\|u(a)\|_{B}<\infty$.

Now we are ready to state the existence results.  For these results, $g_1$ and $g_2$ are required to satisfy technical conditions found in Section $\ref{operator est}$ (specifically equations $(\ref{mal condition})$ and $(\ref{mal 2 2}))$.  We remark that any Mikhlin multiplier bounded below by one will satisfy these two conditions (recall that $f$ is a Mihklin multiplier if $|f^{(k)}(r)|\leq C|r|^{-k}$).  

As expected in these types of arguments, the full result gives unique local solutions provided the parameters satisfy a large collection of inequalities.  Here we state special cases of the full results.  Our first Theorem is a special case of Theorem $\ref{full version 2}$ (see Section \ref{type 1}) and uses the standard product estimate (Proposition $\ref{product est 1}$ in Section $\ref{Besov space}$).
\begin{theorem}\label{old style short}Let $\gamma_1>1$, $\gamma_2>0$, $s_1$, $s_2$ and $p$ be real numbers such that 
$s_2>\gamma_2$, $0<s_2-s_1<\min\{\gamma_1/2, 1\}$ and $\gamma_1\geq s_2-s_1+1+n/p$.   We also assume that $g_1$ and $g_2$ satisfy equations $(\ref{mal condition})$ and $(\ref{mal 2 2})$.  Then for any divergence free $u_0\in B^{s_1}_{p,q}(\mathbb{R}^n)$, there exists a local solution $u$ to the generalized Leray-alpha equation $(\ref{leray1})$, with
\begin{equation*}u\in BC([0,T):B^{s_1}_{p,q}(\mathbb{R}^n))\cap \dot{C}^T_{a;s_2,p,q},
\end{equation*}
where $a=(s_2-s_1)/\gamma_1$.  $T$ can be chosen to be a non-increasing function of $\|u_0\|_{s_1, p,q}$ with $T=\infty$ if $\|u_0\|_{s_1, p, q}$ is sufficiently small.
\end{theorem}

Before stating our second theorem, we remark that this result also holds if the Besov spaces are replaced by Sobolev spaces.  This is not true of the next theorem, which is a special case of the more general Theorem $\ref{full generality}$, and relies on our second product estimate (Proposition $\ref{product est 2}$ in Section $\ref{Besov space}$).
\begin{theorem}\label{special case 1}Let $\gamma_1>1$, $\gamma_2>0$, $s_1$, $s_2$ and $p$ satisfy
\begin{equation*}\aligned 0&<s_2-s_1<\gamma_1/2,
\\ s_1&>\gamma_2-n/p-1,
\\ \gamma_1&\geq 2s_2-s_1-\gamma_2+n/p+1,
\\ n/p&>\gamma_2/2,
\\ s_2&\geq \gamma_2/2.
\endaligned
\end{equation*}
We also assume that $g_1$ and $g_2$ satisfy equations $(\ref{mal condition})$ and $(\ref{mal 2 2})$.  Then for any divergence free $u_0\in B^{s_1}_{p,q}(\mathbb{R}^n)$, there exists a local solution $u$ to the generalized Leray-alpha equation $(\ref{leray1})$, with
\begin{equation*}u\in BC([0,T):B^{s_1}_{p,q}(\mathbb{R}^n))\cap \dot{C}^T_{a;s_2,p,q},
\end{equation*}
where $a=(s_2-s_1)/\gamma_1$.  $T$ can be chosen to be a non-increasing function of $\|u_0\|_{s_1, p, q}$ with $T=\infty$ if $\|u_0\|_{s_1, p, q}$ is sufficiently small.
\end{theorem}

We remark that in the first theorem, $\gamma_2$ can be arbitrarily large, but $s_1>-1$, while in the second theorem $\gamma_2<2n/p$, but for sufficiently large $\gamma_1$ and sufficiently small $\gamma_2$, $s_1>\gamma_2-n/p-1$ can be less than $-1$.  Thus the non-standard product estimate allows us to obtain existence results for initial data with lower regularity, but requires $\gamma_2$ to be small and requires the use of Besov spaces.

We also note that if we set $\gamma_2=0$ and $g_2(\xi)=1$ (and thus are back in the case of the generalized Navier-Stokes equation), then these techniques would recover the results of Wu in \cite{Wu} for the generalized Navier-Stokes equation.

As was stated above, these results will hold if the $g_i$ are Mihklin multipliers.  However, there are interesting choices of $g_i$ (specifically $g_i$ being, essentially, a logarithm) which are not Mihklin multipliers.  The following theorem addresses this case.

\begin{theorem}\label{log cor}Let $g_1$ and $g_2$ both satisfy 
\begin{equation}\label{cor part}\aligned g_i(r)&\leq Cr^{\varepsilon},
\\ |g^{(k)}_i(r)|&\leq Cr^{-k},
\endaligned
\end{equation}
for any $\varepsilon>0$ and for $1\leq k\leq n/2+1$.   Then Theorem's $\ref{old style short}$ and $\ref{special case 1}$ will hold, provided  $\gamma_1$ is replaced in the restrictions on the parameters by $\gamma_1-\varepsilon$ for arbitrarily small $\varepsilon>0$.
\end{theorem}
This follows directly from the results in Section $\ref{operator est}$, specifically Proposition $\ref{annie}$ and Lemma $\ref{heat kernel bound log}$.

Incorporating the additional constraints from the energy bound in \cite{kazuo}, we can now state the global existence result.
\begin{corollary} Let $p=2$ and let $n=3$.  Then, for any of our local existence results, if we additionally assume that  
\begin{equation*}\aligned &2\gamma_1+\gamma_2\geq 5,
\\  &\int_1^\infty \frac{ds}{s g_1^2(s)g_2(s)} =\infty,
\endaligned
\end{equation*}
then the local solutions can be extended to global solutions.
\end{corollary}
Note that if $g_1$ and $g_2$ are Mihklin multipliers, then the additional constraint on the $g_i$ is satisfied.  Also note that if $g_1$ and $g_2$ satisfy $(\ref{cor part})$, then $g_1(s)g_2(s)\leq Cs^{\varepsilon}$ for any $\varepsilon>0$ (at least almost everywhere), which is similar to the requirement from Theorem \ref{log cor}.

The proof of the corollary relies on the smoothing effect of the operator $e^{t\mathcal{L}_1}$, which ensures that, for any $t>0$, our local solution $u(t,\cdot)\in B^r_{2,q}(\mathbb{R}^3)$ for any $r\in \mathbb{R}$.  This provides the smoothness necessary to use the energy bound from \cite{kazuo} to obtain a uniform-in-time bound on the $B^{s_1}_{2,q}(\mathbb{R}^3)$ norm of the solution, and then a standard bootstrapping argument completes the proof of global existence.  In Section $\ref{Higher regularity for the local existence result}$, we include an argument detailing this smoothing effect for the solution to Theorem $\ref{old style short}$.  

Finally, we remark that extending the local solutions to global solutions for $p\neq 2$ and $n>3$ will be the subject of future work.  Handling $n>3$ should follow by tweaking the argument used in \cite{kazuo}.  Obtaining global solutions for $p\neq 2$ is significantly more complicated, and the argument will follow the interpolation based argument used by Gallagher and Planchon in \cite{galplan} for the two dimensional Navier-Stokes equation.

\section{Besov spaces}\label{Besov space}

We begin by defining the Besov spaces $B^s_{p,q}(\mathbb{R}^n)$.  Let $\psi_0$ be an even, radial, Schwartz function with Fourier transform $\hat{\psi_0}$ that has the following properties:
\begin{equation*}\aligned &\hat{\psi_0}(x)\geq 0
\\  &\text{support~}\hat{\psi_0}\subset A_0:=\{\xi\in \mathbb{R}^n:2^{-1}<|\xi|<2\}
\\ &\sum_{j\in\mathbb{Z}} \hat{\psi_0}(2^{-j}\xi)=1, ~\text{for all}~ \xi\neq 0.
\endaligned
\end{equation*}

We then define 
$\hat{\psi_j}(\xi)=\hat{\psi}_0(2^{-j}\xi)$ (from Fourier inversion, this also means $\psi_j(x)=2^{jn}\psi_0(2^jx)$), and remark that $\hat{\psi_j}$ is supported in $A_j:=\{\xi\in\mathbb{R}^n:2^{j-1}<|\xi|<2^{j+1}\}$.  We also define $\Psi$ by 
\begin{equation}\label{low freq part}\hat{\Psi}(\xi)=1-\sum_{k=0}^\infty \hat{\psi}_k(\xi).
\end{equation}

We define the Littlewood Paley operators $\triangle_j$ and $S_j$ by
\begin{equation*}\triangle_j f=\psi_j\ast f, \quad
S_jf=\sum_{k=-\infty}^{j}\triangle_k f,
\end{equation*}
and record some properties of these operators.  Applying the Fourier Transform and
recalling that $\hat{\psi}_j$ is supported on $2^{j-1}\leq
|\xi|\leq2^{j+1}$, it follows that   
\begin{equation}\aligned \label{besovlemma1}\triangle_j\triangle_k f= 0, \quad |j-k|\geq 2
\\ \triangle_j (S_{k-3}f\triangle_{k}g)= 0 \quad |j-k|\geq 4,
\endaligned
\end{equation}
and, if $|i-k|\leq 2$, then 
\begin{equation}\label{besovpieces67}\triangle_j(\triangle_kf\triangle_i g)=0 \quad j>k+4.
\end{equation}

For $s\in\mathbb{R}$ and $1\leq p,q\leq \infty$ we define
the space $\tilde{B}^s_{p,q}(\mathbb{R}^n)$ to be the set of distributions such that 
\begin{equation*}\|u\|_{\tilde{B}^s_{p,q}}=\(\sum_{j=0}^\infty (2^{js}\|\triangle_j
u\|_{L^p})^q\)^{1/q}<\infty,
\end{equation*}
with the usual modification when $q=\infty$.  Finally, we define the Besov spaces $B^s_{p,q}(\mathbb{R}^n)$ by the norm 
\begin{equation*}\|f\|_{B^s_{p,q}}=\|\Psi*f\|_p+\|f\|_{\tilde{B}^s_{p,q}},
\end{equation*}
for $s>0$.  For $s>0$, we define $B^{-s}_{p',q'}$ to be the dual
of the space $B^s_{p,q}$, where $p',q'$ are the Holder-conjugates to
$p,q$.

These Littlewood-Paley operators are also used to define Bony's paraproduct.  We have 
\begin{equation}\label{lp start}fg=\sum_{k} S_{k-3}f\triangle_k g + \sum_{k}S_{k-3}g\triangle_k f+ \sum_{k}\triangle_k f\sum_{l=-2}^2 \triangle_{k+l} g.
\end{equation}

The estimates $(\ref{besovlemma1})$ and $(\ref{besovpieces67})$ imply that 
\begin{equation}\aligned \label{bony256}\triangle_j (fg)\leq &\sum_{k=-3}^3 \triangle_j (S_{j+k-3}f\triangle_{j+k} g)+ \sum_{k=-3}^3 \triangle_j (S_{j+k-3}g\triangle_{j+k} f)
\\ +&\sum_{k>j-4}\triangle_j \(\triangle_k f\sum_{l=-2}^2 \triangle_{k+l}g\).
\endaligned
\end{equation}

Now we turn our attention to establishing some basic Besov space estimates.  First, we let $1\leq q_1\leq q_2\leq \infty$, $\beta_1\leq \beta_2$, $1\leq p_1\leq p_2\leq\infty$, $\gamma_1=\gamma_2+n(1/p_1-1/p_2)$, and $r>s>0$.  Then we have the following:
\begin{equation}\label{besov embedding}\aligned \|f\|_{B^{\beta_1}_{p,q_2}}&\leq C\|f\|_{B^{\beta_2}_{p,q_1}},
\\ \|f\|_{B^{\gamma_2}_{p_2,q}}&\leq C\|f\|_{B^{\gamma_1}_{p_1,q}},
\\ \|f\|_{H^{s,p}}&\leq \|f\|_{B^r_{p,q}}, 
\\ \|f\|_{H^{s,2}}&=\|f\|_{B^s_{2,2}}\leq \|f\|_{B^{r}_{2,q}}.
\endaligned
\end{equation}
These will be referred to as the Besov embedding results.  

Next we record our two different Leibnitz-rule type estimate.  The first is the standard estimate, which can be found in (among many other places) Lemma $2.2$ in \cite{chae}.  See also  Proposition $1.1$ in \cite{TT}.
\begin{proposition}\label{product est 1}Let $s>0$ and $q\in [1,\infty]$.  Then 
\begin{equation*}\|fg\|_{B^s_{p,q}}\leq C\(\|f\|_{L^{p_1}}\|g\|_{B^s_{p_2,q}}+\|f\|_{B^s_{q_1,q}}\|g\|_{L^{q_2}}\),
\end{equation*}
where $1/p=1/p_1+1/p_2=1/q_1+1/q_2$.
\end{proposition}

Our second product estimate is less common.  The estimate originated in \cite{chemin}; another proof can be found in \cite{besovpaper2}.
\begin{proposition}\label{product est 2} Let $f\in B^{s_1}_{p_1,q}(\mathbb{R}^n)$ and let $g\in B^{s_2}_{p_2,q}(\mathbb{R}^n)$.  Then, for any $p$ such that $1/p\leq 1/p_1+1/p_2$ and with $s=s_1+s_2-n(1/p_1+1/p_2-1/p)$,  we have 
\begin{equation*}\|fg\|_{B^s_{p,q}}\leq \|f\|_{B^{s_1}_{p_1,q}}\|g\|_{B^{s_2}_{p_2,q}},
\end{equation*}
provided $s_1<n/p_1$, $s_2<n/p_2$, and $s_1+s_2>0$.
\end{proposition}

\section{Local Existence by Proposition $\ref{product est 1}$}\label{type 1}
Our goal in this section is to prove the following Theorem.
\begin{theorem}\label{full version 2}Let $\gamma_1>1$, $\gamma_2>0$, and assume $g_1$ and $g_2$ satisfy equations $(\ref{mal condition})$ and $(\ref{mal 2 2})$.  Let $u_0\in B^{s_1}_{p,q}(\mathbb{R}^n)$ be divergence-free.  Then there exists a local solution $u$ to the generalized Leray-alpha equation $(\ref{leray1})$, with
\begin{equation*}u\in BC([0,T):B^{s_1}_{p,q}(\mathbb{R}^n))\cap \dot{C}^T_{a;s_2,p,q},
\end{equation*}
where $a=(s_2-s_1)/\gamma_1$ if there exists $k>0$ such that the parameters satisfy $(\ref{full 2 list})$.  $T$ can be chosen to be a non-increasing function of $\|u_0\|_{s_1, p, q}$ with $T=\infty$ if $\|u_0\|_{s_1, p, q}$ is sufficiently small.
\end{theorem}

We begin by re-writing equation $(\ref{leray1})$ as 
\begin{equation}\aligned \label{leray2}&\partial_t u+P(1-\alpha^2\ltwo)^{-1}\div(u\tensor (1-\alpha^2\ltwo)u)-\nu\lone u=0,
\\ &u(0,x)=u_0(x), ~~ \div u_0=\div u=0,
\endaligned
\end{equation}
where $P$ is the Hodge projection onto divergence free vector fields and an application of the divergence free condition shows $\nabla_u (1-\alpha^2  \ltwo) u=\div(u\tensor (1-\alpha^2\ltwo)u)$, where $v\tensor w$ is the matrix with $ij$ entry equal to the product of the $i^{th}$ coordinate of $v$ and the $j^{th}$ coordinate of $w$.

Setting $\alpha=1$ for notational simplicity and applying Duhamel's principle, we get that $u$ is a solution to the equation if and only if $u$ is a fixed point of the map $\Phi$ given by 

\begin{equation*}\Phi(u)=e^{t\mathcal{L}_1}u_0+\int_0^t e^{(t-s)\mathcal{L}_1}(W(u(s),u(s)))ds,
\end{equation*}
where $W(u,v)=(1-\mathcal{L}_2)^{-1}\div(u(s)\tensor (1+\mathcal{L}_2)v(s))$.  Our goal is to show that $\Phi$ is a contraction in the space 
\begin{equation*}\aligned X_{T,M}=\{f\in BC([0,T):B^{s_1}_{2,q}(\mathbb{R}^n))\cap \dot{C}_{a;s_2,p,q} \text{~and~}
\\ \sup_t \|f(t)-e^{t\mathcal{L}_1u_0}\|_{B^{s_1}_{p,q}}+\sup_t t^{a}\|u(t)\|_{B^{s_2}_{p,q}}<M\},
\endaligned
\end{equation*}
where $a=(s_2-s_1)/\gamma_1$, for strictly positive $T$ and $M$ to be chosen later.

Following the arguments outlined in \cite{KP} and \cite{sobpaper}, $\Phi$ will be a contraction if we can show that 
\begin{equation}\label{main parts 1}\aligned &\sup_{t}t^{a}\|e^{t\mathcal{L}_1}u_0\|_{B^{s_2}_{2,q}}<M/3,
\\ &\sup_{t}\|\int_0^t e^{(t-s)\mathcal{L}_1} W(u(s),u(s))ds\|_{B^{s_1}_{2,q}}<M/3,
\\ &\sup_{t}t^a\|\int_0^t e^{(t-s)\mathcal{L}_1} W(u(s),u(s))ds\|_{B^{s_2}_{2,q}}<M/3,
\endaligned
\end{equation}
for $u\in X_{T,M}$.  Similar to Proposition $3$ in \cite{sobpaper}, Proposition $\ref{heat kernel bound}$ and the definition of $a$ give 
\begin{equation*}\sup_{t}t^{a}\|e^{t\mathcal{L}_1}u_0\|_{B^{s_2}_{2,q}}<M/3
\end{equation*}
holds a small enough choice of $T>0$.  Turning to the second inequality, applying Minkowski's inequality and Proposition $\ref{heat kernel bound}$, we get  
\begin{equation}\label{23-1}\aligned &\sup_{t}\|\int_0^t e^{(t-s)\Li} W(u(s),u(s))ds\|_{B^{s_1}_{p,q}}
\\ \leq &\sup_{t}\int_0^t \|e^{(t-s)\Li} W(u(s),u(s))\|_{B^{s_1}_{p,q}}ds
\\ \leq &\sup_{t}\int_0^t |t-s|^{-(s_1-r+n/p^*-n/p)/\gamma_1} \|W(u(s),u(s))\|_{B^{r}_{p^*,q}}ds,
\endaligned
\end{equation}
where $p^*\leq p$ will be specified later.  Using Proposition $\ref{product est 1}$, we have 
\begin{equation}\label{2mal}\aligned &\|W(u(s),u(s))\|_{B^{r}_{p^*,q}}\leq \|u\tensor (1+\mathcal{L}_2)u\|_{B^{r+1-\gamma_2}_{p^*,q}}
\\ \leq &\|u\|_{L^{p_1}}\|(1+\mathcal{L}_2)u\|_{B^{r+1-\gamma_2}_{p_2,q}}+\|u\|_{B^{r+1-\gamma_2}_{q_1,1}}\|(1+\mathcal{L}_2)u\|_{L^{q_2}}
\\ \leq&\|u\|_{L^{p_1}}\|u\|_{B^{r+1}_{p_2,q}}+\|u\|_{B^{r+1-\gamma_2}_{q_1,1}}\|(1+\mathcal{L}_2)u\|_{L^{q_2}},
\endaligned 
\end{equation}
where $1/p^*=1/p_1+1/p_2=1/q_1+1/q_2$, provided $r+1-\gamma_2>0$ (note that if $\gamma_2=0$, the choice of $r=-1$ reduces this to Holder's inequality).  In order to complete the argument, we need to bound this by $\|u\|_{B^{s_2}_{p,q}}^2$.  To facilitate this, we choose $r+1=s_2$ (which forces $s_2>\gamma_2$) and  $q_2=p_2=p$ (which forces $p_1=q_1$) and then equation $(\ref{2mal})$ becomes  
\begin{equation}\label{2mali}\aligned \|W(u(s),u(s))\|_{B^{r}_{p^*,q}}\leq& \|u\|_{L^{p_1}}\|u\|_{B^{r+1}_{p_2,q}}+\|u\|_{B^{r+1-\gamma_2}_{q_1,1}}\|(1+\mathcal{L}_2)u\|_{L^{q_2}}
\\ \leq&\|u\|_{B^{s_2}_{p,q}}\(\|u\|_{L^{p_1}}+\|u\|_{B^{r+1-\gamma_2}_{p_1,1}}\),
\endaligned 
\end{equation}
where we used Proposition $\ref{annie}$ for the second inequality.
Finally, choosing $p_1=np/(n-kp)$ for some $k<\gamma_2$, we use a Sobolev embedding estimate (see Proposition $6.4$ in \cite{T3}) to get 
\begin{equation}\|W(u(s),u(s))\|_{B^{r}_{p^*,q}}\leq \|u\|_{B^{s_2}_{p,q}}^2,
\end{equation}
provided 
\begin{equation}\label{product 1 term}\aligned 1/p^*-1/p&=1/p_1=(n-k p)/np,
\\ s_2&>\gamma_2,
\\ s_2&=r+1,
\\ kp&<n.
\endaligned
\end{equation}

Returning to the estimate begun in $(\ref{23-1})$, we have 
\begin{equation*}\aligned  &\sup_{t}\|\int_0^t e^{(t-s)\Li} W(u(s),u(s))ds\|_{B^{s_1}_{p,q}}
\\ \leq &\sup_t \int_0^t |t-s|^{-(s_1-r+n/p^*-n/p)/\gamma_1}s^{-2a}s^{2a}\|u(s)\|^2_{B^{s_2}_{p,q}}ds
\\ \leq &C\sup_t \|u\|^2_{a;s_2,p,q}t^{-(s_1-r+n/p^*-n/p)/\gamma_1-2(s_2-s_1)/\gamma_1+1}<M/3,
\endaligned
\end{equation*}
provided 
\begin{equation}\label{2ty1}\aligned 0&\leq (s_1-r+n/p^*-n/p)/\gamma_1<1
\\ 1&>2(s_2-s_1)/\gamma_1
\\0&\leq -(s_1-r+n/p^*-n/p)/\gamma_1-2(s_2-s_1)/\gamma_1+1.
\endaligned
\end{equation}

For the last term in $(\ref{2main parts 1})$, we have 
\begin{equation*}\aligned &\sup_{t}t^{a}\|\int_0^t e^{(t-s)\Li} W(u(s),u(s))ds\|_{B^{s_2}_{p,q}}
\\ \leq &\sup_{t}t^{a}\int_0^t |t-s|^{-(s_2-r+n/p^*-n/p)/\gamma_1} \|W(u(s),u(s))\|_{B^{r}_{2,q}}ds
\\ \leq &\sup_{t}t^{a}\int_0^t |t-s|^{-(s_2-r+n/p^*-n/p)/\gamma_1}s^{-2a}s^{2a}\|u(s)\|^2_{B^{s_2}_{p,q}}ds
\\ \leq &C\|u\|_{B^{s_2}_{p,q}}^2\sup_t t^{a}t^{-(s_2-r+n/p^*-n/p)/\gamma_1-2(s_2-s_1)/\gamma_1+1}
\\ \leq &CM^2t^{-(s_2-r)/\gamma_1-(s_2-s_1)/\gamma_1+1}<M/3,
\endaligned
\end{equation*}
provided 
\begin{equation}\label{2ty2}\aligned 0&\leq (s_2-r+n/p^*-n/p)/\gamma_1<1,
\\ 1&>2(s_2-s_1)/\gamma_1,
\\ 0&\leq -(s_2-r+n/p^*-n/p)/\gamma_1-(s_2-s_1)/\gamma_1+1.
\endaligned
\end{equation}

Combining $(\ref{2ty1})$, and $(\ref{2ty2})$ (and removing redundancies) gives 
\begin{equation*}\aligned s_1&>r
\\ \gamma_1/2&>s_2-s_1,
\\0&\leq s_2-r+n/p^*-n/p<\gamma_1,
\\ \gamma_1&\geq 2s_2-r+n/p^*-n/p-s_1.
\endaligned
\end{equation*}
Incorporating $(\ref{product 1 term})$, and observing that the last inequality implies the third in the preceding list of inequalities, we get 
\begin{equation}\label{full 2 list}\aligned s_2&>\gamma_2\geq k,
\\ kp&<n,
\\ s_2-s_1&<\min\{\gamma_1/2, 1\},
\\ \gamma_1&\geq s_2-s_1+1+n/p-k.
\endaligned
\end{equation}
This completes Theorem $\ref{full version 2}$.  Note that for $\gamma_1=2$ and $\gamma_2=0$, this recovers, up to the slight modification mentioned after equation $(\ref{2mal})$, the result from \cite{kp}  for the Navier-Stokes equation.  In general, the distinction between this result and the one in the next section is that $s_1$ must be larger, but there is no bound on the size of $\gamma_2$.  To get Theorem $\ref{old style short}$, choose $k$ to be an arbitrarily small positive number, and the last inequality becomes $\gamma_1\geq s_2-s_1+1+n/p$.

\section{Local Existence using Proposition $\ref{product est 2}$}\label{type 2}
In this section we prove the following local existence result.
\begin{theorem}\label{full generality}Let $\gamma_1>1$, $\gamma_2>0$, and assume $g_1$ and $g_2$ satisfy equations $(\ref{mal condition})$ and $(\ref{mal 2 2})$.  Let $u_0\in B^{s_1}_{p,q}(\mathbb{R}^n)$ be divergence-free.  Then there exists a local solution $u$ to the generalized Leray-alpha equation $(\ref{leray1})$, with
\begin{equation*}u\in BC([0,T):B^{s_1}_{p,q}(\mathbb{R}^n))\cap \dot{C}^T_{a;s_2,p,q},
\end{equation*}
where $a=(s_2-s_1)/\gamma_1$, if there exists $r$, $r_1$ and $r_2$ such that all the parameters satisfy $(\ref{final list 2})$.  $T$ can be chosen to be a non-increasing function of $\|u_0\|_{s_1, p, q}$ with $T=\infty$ if $\|u_0\|_{s_1, p, q}$ is sufficiently small.
\end{theorem}

With the same set-up as the previous section, our goal is to show that 
\begin{equation}\label{2main parts 1}\aligned &\sup_{t}t^{a}\|e^{t\mathcal{L}_1}u_0\|_{B^{s_2}_{2,q}}<M/3,
\\ &\sup_{t}\|\int_0^t e^{(t-s)\mathcal{L}_1} W(u(s),u(s))ds\|_{B^{s_1}_{2,q}}<M/3,
\\ &\sup_{t}t^a\|\int_0^t e^{(t-s)\mathcal{L}_1} W(u(s),u(s))ds\|_{B^{s_2}_{2,q}}<M/3.
\endaligned
\end{equation}
The first inequality follows exactly as it did in the previous section.

For the second, using Minkowski's inequality and Proposition $\ref{heat kernel bound}$, we have 
\begin{equation}\label{3-1}\aligned &\sup_{t}\|\int_0^t e^{(t-s)\Li} W(u(s),u(s))ds\|_{B^{s_1}_{p,q}}
\\ \leq &\sup_{t}\int_0^t \|e^{(t-s)\Li} W(u(s),u(s))\|_{B^{s_1}_{p,q}}ds
\\ \leq &\sup_{t}\int_0^t |t-s|^{-(s_1-r)/\gamma_1} \|W(u(s),u(s))\|_{B^{r}_{p,q}}ds,
\endaligned
\end{equation}
where $r\leq s_1$ and will be specified later.  Using Proposition $\ref{product est 2}$ and Proposition $\ref{annie}$, we have 
\begin{equation}\label{mal}\aligned &\|W(u(s),u(s))\|_{B^{r}_{p,q}}\leq \|u\tensor (1+\mathcal{L}_2)u\|_{B^{r+1-\gamma_2}_{p,q}}
\\ \leq &\|u\|_{B^{r_1}_{p,q}}\|(1+\mathcal{L}_2)u\|_{B^{r_2}_{p,q}}\leq \|u\|_{B^{r_1}_{p,q}}\|u\|_{B^{r_2+\gamma_2}_{p,q}}\leq \|u\|_{B^{s_2}_{p,q}},
\endaligned 
\end{equation}
provided 
\begin{equation}\label{aldi}\aligned r+1-\gamma_2&\leq r_1+r_2-n/p,
\\ r_1+r_2&>0,
\\ r_1, r_2&<n/p,
\\ s_2&\geq \max\{r_1, r_2+\gamma_2\}.
\endaligned
\end{equation}

Returning to equation $(\ref{3-1})$, we have 
\begin{equation}\aligned  &\sup_{t}\|\int_0^t e^{(t-s)\Li} W(u(s),u(s))ds\|_{B^{s_1}_{p,q}}
\\ \leq &\sup_t \int_0^t |t-s|^{-(s_1-r)/\gamma_1}s^{-2a}s^{2a}\|u(s)\|^2_{B^{s_2}_{p,q}}ds
\\ \leq &C\sup_t \|u\|^2_{a;s_2,p,q}t^{-(s_1-r)/\gamma_1-2(s_2-s_1)/\gamma_1+1}<M/3,
\endaligned
\end{equation}
provided 
\begin{equation}\label{ty1}\aligned 0&\leq (s_1-r)/\gamma_1<1
\\ 1&>2(s_2-s_1)/\gamma_1
\\0&\leq -(s_1-r)/\gamma_1-2(s_2-s_1)/\gamma_1+1.
\endaligned
\end{equation}

Estimating the last term of $(\ref{main parts 1})$ in a similar fashion, we have 
\begin{equation}\aligned &\sup_{t}t^{a}\|\int_0^t e^{(t-s)\Li} W(u(s),u(s))ds\|_{B^{s_2}_{p,q}}
\\ \leq &\sup_{t}t^{a}\int_0^t |t-s|^{-(s_2-r)/\gamma_1} \|W(u(s),u(s))\|_{B^{r}_{2,q}}ds
\\ \leq &\sup_{t}t^{a}\int_0^t |t-s|^{-(s_2-r)/\gamma_1}s^{-2a}s^{2a}\|u(s)\|^2_{B^{s_2}_{p,q}}ds
\\ \leq &C\|u\|_{B^{s_2}_{p,q}}^2\sup_t t^{a}t^{-(s_2-r)/\gamma_1-2(s_2-s_1)/\gamma_1+1}
\\ \leq &CM^2t^{-(s_2-r)/\gamma_1-(s_2-s_1)/\gamma_1+1}<M/3,
\endaligned
\end{equation}
provided 
\begin{equation}\label{ty2}\aligned 0&\leq (s_2-r)/\gamma_1<1,
\\ 1&>2(s_2-s_1)/\gamma_1,
\\ 0&\leq -(s_2-r)/\gamma_1-(s_2-s_1)/\gamma_1+1.
\endaligned
\end{equation}

Our final task is to unify the conditions on the parameters.  The sets of inequalities from equations $(\ref{ty1})$ and $(\ref{ty2})$ can be simplifed to 
\begin{equation}\aligned 0&<s_2-s_1<\gamma_1/2,
\\ s_1&\geq r >s_2-\gamma_1,
\\ \gamma_1&\geq (s_2-s_1)+(s_2-r).
\endaligned
\end{equation}
Incorporating the inequalities from $(\ref{aldi})$, we have 
\begin{equation}\label{vague list}\aligned 0&<s_2-s_1<\gamma_1/2,
\\ s_1&\geq r >s_2-\gamma_1,
\\ \gamma_1&\geq (s_2-s_1)+(s_2-r),
\\ r&+1-\gamma_2\leq r_1+r_2-n/p,
\\ r_1&+r_2>0,
\\ r_1&, r_2<n/p,
\\ s_2&\geq \text{max}\{r_1,r_2+\gamma_2\},
\endaligned
\end{equation}
and this completes the proof of Theorem $\ref{full generality}$.  To obtain the results in Theorem $\ref{special case 1}$, we fix the values of the parameters $r_1, r_2,$ and $r$ in the following way.  First, since our primary interest is in minimizing $s_1$ and $s_2$, we see from the last inequality that this is helped by minimizing $\max\{r_1, r_2+\gamma_2\}$, subject to the constraints $r_1+r_2>0$ and $r_1, r_2<n/p$.  This is accomplished by choosing $r_2=-\gamma_2/2$ and $r_1=\gamma_2/2+M$, where $M$ is some positive number.  Additionally choosing the fourth inequality in the list $(\ref{vague list})$ to be an equality, the list $(\ref{vague list})$ becomes 
\begin{equation}\label{vague list 2}\aligned 0&<s_2-s_1<\gamma_1/2,
\\ s_1&\geq r >s_2-\gamma_1,
\\ \gamma_1&\geq (s_2-s_1)+(s_2-r),
\\ r&=-1+\gamma_2+M-n/p,
\\ n/p&>\gamma_2/2+M,
\\ s_2&\geq \gamma_2/2+M.
\endaligned
\end{equation}
Using the equality in line $4$ to eliminate $r$ from the other inequalities, and then removing extraneous inequalities, we finally get 
\begin{equation}\label{final list 2}\aligned 0&<s_2-s_1<\gamma_1/2,
\\ s_1&\geq \gamma_2+M-n/p-1,
\\ \gamma_1&\geq 2s_2-s_1-\gamma_2-M+n/p+1,
\\ n/p&>\gamma_2/2+M,
\\ s_2&\geq \gamma_2/2+M.
\endaligned
\end{equation}
Eliminating the free parameter $M$ weakens this to
\begin{equation}\label{final list 3}\aligned 0&<s_2-s_1<\gamma_1/2,
\\ s_1&>\gamma_2-n/p-1,
\\ \gamma_1&\geq 2s_2-s_1-\gamma_2+n/p+1,
\\ n/p&>\gamma_2/2,
\\ s_2&\geq \gamma_2/2,
\endaligned
\end{equation}
which finishes Theorem $\ref{special case 1}$.

\section{Operator estimate for $e^{t\mathcal{L}}$}\label{operator est}

We define the Fourier multiplier $\mathcal{L}^g_\gamma$ by 
\begin{equation*}\mathcal{L}^g_\gamma u(x)=\int -\frac{|\xi|^{\gamma}}{g(|\xi|)}\hat{u}(\xi)e^{ix\cdot \xi}d\xi,
\end{equation*}
where $\gamma\in \mathbb{R}$ and $g:\mathbb{R}\rightarrow \mathbb{R}$ is nondecreasing and bounded below by 1. We begin this section with a straightforward calculation.

\begin{proposition}\label{annie} Let $1<p<\infty$ and let $G$ be the operator with symbol $1/g$.  Then $\mathcal{L}^g_\gamma: H^{s_1+\gamma,p}(\mathbb{R}^n)\rightarrow H^{s_1,p}(\mathbb{R}^n)$, with 
\begin{equation*}\|\mathcal{L}^g_\gamma f\|_{H^{s_1,p}}\leq C\|f\|_{H^{s_1+\gamma,p}},
\end{equation*}
provided $G$ is $L^p(\mathbb{R}^n)$ bounded.
\end{proposition}
The proof follows directly from the definition of $\mathcal{L}^g_\gamma$, since we have 
\begin{equation*}\|\mathcal{L}^g_\gamma f\|_{H^{s_1,p}}= \|G \((-\triangle)^{\gamma/2} f\)\|_{H^{s_1,p}}\leq \|f\|_{H^{s_1+\gamma,p}},
\end{equation*}
given the assumption on $G$.

Note that if $g$ is bounded below by $1$ and $|g^{(k)}(r)|\leq Cr^{-k}$, for $1\leq k\leq n/2+1$, then G will satisfy the stated condition.  

Now we turn to the main topic of this section, the operator $e^{t\mathcal{L}^g_\gamma}$, which is defined to be the Fourier multiplier with symbol $e^{-t|\xi|^\gamma/g(|\xi|)}$.  The goal in this section is to establish operator bounds for $e^{t\mathcal{L}_\gamma}$ in the case where $\gamma>1$, and following the general outline of the same task for $e^{t\triangle}$, we need to first establish $L^p-L^q$ boundedness for the operator.

\begin{proposition}\label{lpsemi} Let $1\leq p\leq q\leq \infty$, and assume $\gamma>1$.  Then 
\begin{equation*}e^{t\mathcal{L}^g_\gamma}:L^p(\mathbb{R}^n)\rightarrow L^q(\mathbb{R}^n),
\end{equation*}
and we have the bound 
\begin{equation}\label{mal2}\|e^{t\mathcal{L}^g_\gamma}f\|_{L^q}\leq Ct^{-(n/p-n/q)/\gamma}\|f\|_{L^p},
\end{equation}
provided 
\begin{equation}\label{mal condition}\sup_{t\in (0, 1)} \left \lvert \int_0^\infty \partial_r^{n+1}\(r^{n-1}e^{-r^\gamma/g(rt^{-1/\gamma})}\) dr\right \rvert<\infty.
\end{equation}
\end{proposition}

\begin{proof}
For notational convenience, we will suppress the subscript $\gamma$ and the superscript $g$ on the operator for the duration of the proof.  We first prove the result in the special case that $g$ is a constant function and, without loss of generality, further assume the constant is one.  Setting $e^{t\mathcal{L}}f=e^{t\mathcal{L}}\delta\ast f$, where the Fourier Transform of $e^{t\mathcal{L}}\delta(x)$ is equal to $e^{-t|\xi|^\gamma}$, and applying  Young's inequality, we get that 
\begin{equation*}\|e^{t\mathcal{L}}f\|_{L^q}\leq \|e^{t\mathcal{L}}\delta\|_{L^r}\|f\|_{L^p},
\end{equation*}
where $1+1/q=1/r+1/p$.  Formally, we have that 
\begin{equation*}e^{t\mathcal{L}}\delta(\xi)=C\int_{\mathbb{R}^n}e^{-t|x|^\gamma}e^{ix\cdot\xi}dx.
\end{equation*}
Making the variable change $x\rightarrow t^{1/\gamma}x$, we get 
\begin{equation*}e^{t\mathcal{L}}\delta(\xi)=Ct^{-n/\gamma}\int_{\mathbb{R}^n}e^{-|x|^\gamma} e^{it^{-1/\gamma}x\cdot\xi}dx.
\end{equation*}
Taking the $L^r(\mathbb{R}^n)$ norm gives 
\begin{equation*}\|e^{t\mathcal{L}}\delta\|_{L^r}=Ct^{-n/\gamma}\(\int_{\mathbb{R}^n}\left \lvert\int_{\mathbb{R}^n}e^{-|x|^\gamma} e^{it^{-1/\gamma}x\cdot\xi}dx \right \rvert^rd\xi\)^{1/r}.
\end{equation*}
Making the variable change $\xi\rightarrow t^{-1/\gamma}\xi$, this finally becomes 
\begin{equation*}\|e^{t\mathcal{L}}\delta\|_{L^r}=Ct^{-n/\gamma+n/(r\gamma)}\(\int_{\mathbb{R}^n}\left \lvert\int_{\mathbb{R}^n}e^{-|x|^\gamma} e^{ix\cdot\xi}dx \right \rvert^rd\xi\)^{1/r}.
\end{equation*}

Since $1-1/r=1/p-1/q$, the only remaining task is to show that  the integral is finite.  Changing to polar coordinates and bounding the angular portions by a constant, we have 
\begin{equation*}\(\int_{\mathbb{R}^n} \babsl \int_{\mathbb{R}^n}e^{-|x|^\gamma} e^{ix\cdot\xi}dx \babsr^rd\xi\)^{1/r}\leq C\(\int_{0}^\infty \tau_1^{n-1}\babsl\int_{0}^\infty \tau_2^{n-1}e^{-\tau_2^\gamma} e^{i\tau_1\tau_2}d\tau_2\babsr^rd\tau_1\)^{1/r}.
\end{equation*}

The integral over the region where $\tau_1$ is between $0$ and $1$ is clearly finite, so we only have left to consider the $\tau_1$ integral over the region $[1,\infty)$.  To that end, we integrate by parts $k$ times and get 
\begin{equation*}\aligned &\(\int_{1}^\infty \tau_1^{n-1}\babsl\int_{0}^\infty \tau_2^{n-1}e^{-\tau_2^\gamma} \tau_1^{-k}\partial_{\tau_2}^{(k)}e^{i\tau_1\tau_2}d\tau_2 \babsr^rd\tau_1\)^{1/r}
\\ \leq &C\(\int_{1}^\infty \tau_1^{n-1-kr}(C+\babsl\int_{0}^\infty  \partial_{\tau_2}^{(k)} \(\tau_2^{n-1}e^{-\tau_2^\gamma}\) d\tau_2\babsr^rd\tau_1)\)^{1/r}.
\endaligned
\end{equation*}

The integral in $\tau_1$ is finite provided $n<kr$, which means $k$ must be at least the smallest integer greater than $n/r$, and setting $k=n+1$ will satisfy this for any choice of $r$.  For the remaining integral, a combinatorial argument shows that the ``worst" term is of the form $\tau_{2}^{n+\gamma-(n+2)}e^{-\tau_2^\gamma}$, which is integrable provided $\gamma>1$, and this concludes the argument for the case when $g$ is a constant function.  

For the general case, following the argument above, the result will follow provided 
\begin{equation}\int_0^\infty \partial_r^{n+1}\(r^{n-1}e^{-r^\gamma/g(rt^{-1/\gamma})}\) dr,
\end{equation}
can be bounded independently of $t$ for $t<1$, which is equation $(\ref{mal condition})$. 

\end{proof}

Before moving on, we address two cases where equation $(\ref{mal condition})$ holds.  
\begin{corollary}\label{kernel bound cor 1} If $g$ satisfies 
\begin{equation}\label{partial mihklin}|g^{(k)} (r)|\leq Cr^{-k},
\end{equation}
for each $0\leq k\leq n+1$, then equation $(\ref{mal condition})$ holds (and thus equation $(\ref{mal2})$ holds for the associated operator $e^{t\mathcal{L}})$.
\end{corollary}
Before starting the proof, we remark that any function $f$ which satisfies the  Mihklin multiplier condition (that $|f^{(k)} (x)|\leq C|x|^{-k}$ holds for all non-negative $k$) and is bounded below by one will satisfy  $(\ref{partial mihklin})$.

\begin{proof}

Computing the derivative in $(\ref{mal condition})$ gives  
\begin{equation}\label{mal condition constant}\int_0^\infty \sum_{k=0}^{n-1}\partial_r^{n-1-k}(r^{n-1})\partial_r^{2+k}\(e^{-r^\gamma/g(rt^{-1/\gamma})}\) dr.
\end{equation}
Computing the derivatives on the exponential results in products whose terms are the exponential, powers of $r$, and derivatives of $\(g(t^{-1/\gamma}r)\)^{-1}$.  While there are more terms in this expansion than in the previous case (where $g$ is assumed to be constant), each term here can be associated to a term in the constant $g$ case.

To help illustrate this, we consider the special case of $n=3$.  Then equation $(\ref{mal condition constant})$ becomes  
\begin{equation*}\int_0^\infty \sum_{k=0}^{2}\partial_r^{2-k}(r^{2})\partial_r^{2+k}\(e^{-r^\gamma/g(rt^{-1/\gamma})}\) dr=\int_0^\infty I_1+I_2+I_3,
\end{equation*}
where 
\begin{equation*}\aligned I_1&=\partial_r^2\(e^{-r^\gamma/g(rt^{-1/\gamma})}\),
\\ I_2&=r\partial_r^3\(e^{-r^\gamma/g(rt^{-1/\gamma})}\),
\\ I_3&=r^2\partial_r^4\(e^{-r^\gamma/g(rt^{-1/\gamma})}\).
\endaligned
\end{equation*}
We begin by considering $I_1$, which can be re-written as 
\begin{equation*}\aligned &I_1=J_1e^{-r^\gamma/g(rt^{-1/\gamma})},
\endaligned
\end{equation*}
where 
\begin{equation*}\aligned J_1&=\(\frac{-r^{\gamma-1}}{g(rt^{-1/\gamma})}+\frac{r^\gamma g'(rt^{-1/\gamma})t^{-1/\gamma}}{(g(rt^{-1\gamma}))^2}\)^2
 +\frac{-r^{\gamma-2}}{g(rt^{-1/\gamma})}
\\ +&\frac{2r^{\gamma-1} g'(rt^{-1/\gamma})t^{-1/\gamma}}{(g(rt^{-1\gamma}))^2}
+ \frac{r^{\gamma} g''(rt^{-1/\gamma})t^{-2/\gamma}}{(g(rt^{-1\gamma}))^2}-  \frac{2r^{\gamma} (g'(rt^{-1/\gamma}))^2t^{-2/\gamma}}{(g(rt^{-1\gamma}))^3}
\endaligned
\end{equation*}

Because $g$ is bounded below by one, we can bound $J_1$ by 
\begin{equation*}\aligned J_1&\leq \(-r^{\gamma-1}+r^\gamma g'(rt^{-1/\gamma})t^{-1/\gamma}\)^2
 +r^{\gamma-2}
\\ +&2r^{\gamma-1} g'(rt^{-1/\gamma})t^{-1/\gamma}
+ r^{\gamma} g''(rt^{-1/\gamma})t^{-2/\gamma}+2r^{\gamma} (g'(rt^{-1/\gamma}))^2t^{-2/\gamma}
\\ =& \(-r^{\gamma-1}+r^{\gamma-1} g'(rt^{-1/\gamma})(rt^{-1/\gamma})\)^2
 +r^{\gamma-2}
\\ +&2r^{\gamma-2} g'(rt^{-1/\gamma})(rt^{-1/\gamma})+ r^{\gamma-2} g''(rt^{-1/\gamma})\(rt^{-/\gamma-2}\)^2 +2r^{\gamma} \(g'(rt^{-1/\gamma})rt^{-1/\gamma}\)^2.
\endaligned
\end{equation*}
Recalling $(\ref{partial mihklin})$, we finally see that $J_1$ satisfies 
\begin{equation*}J_1\leq C\(r^{2(\gamma-1)}+r^{\gamma-1} +r^{\gamma-2}+r^\gamma\),
\end{equation*}
and plugging this back into the equation for $I_1$ gives 
\begin{equation}\aligned I_1\leq &C\(r^{2(\gamma-1)}+r^{\gamma-1} +r^{\gamma-2}+r^\gamma\)e^{-r^\gamma/g(rt^{-1/\gamma})}
\\ \leq &C\(r^{2(\gamma-1)}+r^{\gamma-1} +r^{\gamma-2}+r^\gamma\)e^{-Cr^\gamma},
\endaligned
\end{equation}
where we have again used that $g$ is bounded below by one.  The key observation is that $I_1$ is now bounded by (up to a constant) a term from the constant $g$ case.  $I_2$ and $I_3$ can be bounded similarly, reducing the problem to the constant $g$ case, where the result is already known.  The same argument (though clearly requiring significantly more computations) holds in the general $n$ case, and this concludes the proof.
\end{proof} 

Finally, we consider the case where $g$ is (essentially) the natural log. This is to make our results compatible with the energy bounds for the generalized Leray-$\alpha$ equation established in  \cite{kazuo}.

\begin{lemma}\label{lpsemi g}Let $g(r)\leq Cr^\varepsilon$ for any positive $\varepsilon$ and assume 
\begin{equation}\label{pseudo mihklin}|g^{(m)}(r)|\leq Cr^{-m}
\end{equation}
for any integer $m$ such that $1\leq m\leq n/2+1$ and let $p\leq q$.  Then
\begin{equation*}e^{t\mathcal{L}_\gamma}:L^p(\mathbb{R}^n)\rightarrow L^q(\mathbb{R}^n).
\end{equation*}
For $p=q$, we have 
\begin{equation}\label{mal36}\|e^{t\mathcal{L}}f\|_{L^p}\leq C\|f\|_{L^p},
\end{equation}
and for $p<q$ we have
\begin{equation}\label{mal3}\|e^{t\mathcal{L}}f\|_{L^q}\leq Ct^{-(n/p-n/q)/\gamma-\varepsilon}\|f\|_{L^p},
\end{equation}
for any choice of small $\varepsilon>0$.
\end{lemma}

\begin{proof}
We begin by showing inequality $(\ref{mal3})$.  Since $g$ is not bounded, $g$ does not satisfy equation $(\ref{partial mihklin})$, which prevents applying Corollary $\ref{kernel bound cor 1}$ directly.   By examining the proof of Corollary $\ref{kernel bound cor 1}$, we see that the boundedness of $g$ was only used to remove $g$ from the exponent of the exponential function after differentiating.  So, following the argument in Corollary $\ref{kernel bound cor 1}$, we have that  
\begin{equation}\label{mal789}\int_0^\infty \partial_r^{(n+1)}\(r^{n-1}e^{-r^\gamma/g(t^{-1/\gamma}r)}\) dr
\leq C\int_1^\infty h(r)e^{-r^\gamma/g(t^{-1/\gamma}r)},
\end{equation}
where $h(r)$ is a sum of powers of $r$ (and we used the fact that $h$ is integrable near $0$ to bound the integral between $0$ and $1$).  By assumption, $g(r)\leq Cr^\varepsilon$ for any positive $\varepsilon$, we have 
\begin{equation*}e^{-r^\gamma/g(t^{-1/\gamma}r)}\leq e^{-t^{\varepsilon/\gamma} r^{\gamma-\varepsilon}},
\end{equation*}
and then $(\ref{mal789})$ becomes
\begin{equation*}\int_0^\infty \partial_r^{(n+1)}\(r^{n-1}e^{-r^\gamma/g(t^{-1/\gamma}r)}\) dr
\leq C\int_1^\infty h(r)e^{-t^{\varepsilon/\gamma} r^{\gamma-\varepsilon}}.
\end{equation*}
Making the change of variable $r\rightarrow t^{\frac{\varepsilon}{\gamma(\gamma-\varepsilon)}}$ and then factoring the resulting powers of $t$ out of the integral, we finally get 
\begin{equation*}\int_0^\infty \partial_r^{(n+1)}\(r^{n-1}e^{-r^\gamma/g(t^{-1/\gamma}r)}\) dr
\leq Ct^{-C\varepsilon}\int_1^\infty h(r)e^{r^{\gamma-\varepsilon}}\leq Ct^{-C\varepsilon}.
\end{equation*}
This completes the proof of equation $(\ref{mal3})$.  

Turning our attention to inequality $(\ref{mal36})$, we begin by recalling the Mikhlin multiplier theorem, which states that if the operator $P$, with symbol $p$, satisfies 
\begin{equation}\label{mihklin3}|p^{(k)} (x)|\leq C|x|^{-k},
\end{equation}
for $0\leq k\leq n/2+1$, then $P$ maps $L^p(\mathbb{R}^n)$ into $L^p(\mathbb{R}^n)$ for $1<p<\infty$.  We will show that our symbol satisfies inequality $(\ref{mihklin3})$ uniformly in $t$ for $0<t\leq 1$.  

This clearly holds for the $k=0$ case, since $e^{-t|r|^\gamma/g(r)}\leq 1$.  For $k=1$, we have 
\begin{equation*}\aligned &|\partial_r (e^{-t|r|^\gamma/g(r)})|=\left\lvert|-t\(\frac{r^{\gamma-1}}{g(r)}-\frac{r^\gamma g'(r)}{g(r)^2}\)e^{-t|r|^\gamma/g(r)}\right\rvert
\\ \leq &C\frac{tr^{\gamma-1}}{g(r)}e^{-t|r|^\gamma/g(r)}\leq C\(\frac{tr^{\gamma}}{g(r)}e^{-t|r|^\gamma/g(r)}\)r^{-1}\leq C\(\sup_{x\geq 0} \(xe^{-x}\)\)r^{-1},
\endaligned
\end{equation*}
where we used that $g(r)$ is bounded below by 1 and $(\ref{pseudo mihklin})$.  Since $\sup_{x\geq 0} \(xe^{-x}\)$ is finite, we have finished the $k=1$ case.  Higher $k$ are bounded in an analogous fashion.  Specifically, 
\begin{equation*}\aligned &|(\partial_r)^k (e^{-t|r|^\gamma/g(r)})|\leq C\(\sup_{x\geq 0} \(\sum_{j=1}^k x^je^{-x}\)\)r^{-k}\leq Cr^{-k}.
\endaligned
\end{equation*}
This completes the proof of the Lemma.
\end{proof}

Now that we have established $L^p-L^q$ boundedness for $e^{t\mathcal{L}^g_\gamma}$, we turn our attention to Sobolev space bounds.  We will rely heavily on the following, which is Proposition $7.2$ from Chapter $13$ in \cite{T3}.  
\begin{proposition}\label{taylor prop}Let $e^{tA}$ be a holomorphic semigroup on a Banach space $X$.  Then, for $t>0$, 
\begin{equation*}\|Ae^{tA}f\|_X\leq \frac{C}{t}\|f\|_X,
\end{equation*}
for $0<t\leq 1$.  
\end{proposition}
For our purposes, $A=\mathcal{L}^g_\gamma$ and $X=L^p(\mathbb{R}^n)$.  To use this proposition, we need to know that $e^{t\mathcal{L}^g_\gamma}$ is a holomorphic semigroup, and following the proof of Proposition $7.1$ from Chapter 13 in \cite{T3}, we see that we only need $e^{t\mathcal{L}^g_\gamma}$ to be uniformly bounded from $L^p(\mathbb{R}^n)$ into itself, which we established earlier in this section.  Now we are ready to prove the following.

\begin{proposition}\label{sobolev} Let $1<p<\infty$, $s_1\leq s_2$ and define $G$ to be the Fourier multiplier with symbol $g$.  If $G$ satisfies  
\begin{equation}\label{mal 2 2}\|G f\|_{L^p}\leq C\|f\|_{L^p},
\end{equation} for all $f\in L^p(\mathbb{R}^n)$, then $e^{t\mathcal{L}^g_\gamma}: H^{s_1,p}(\mathbb{R}^n)\rightarrow H^{s_2,p}(\mathbb{R}^n)$ and 
\begin{equation}\|e^{t\mathcal{L}^g_\gamma} f\|_{H^{s_2,p}}\leq t^{-(s_2-s_1)/\gamma}\|f\|_{H^{s_1,p}}.
\end{equation}
\end{proposition}

We first establish this result in the case $s_2=\gamma$ and $s_1=0$.  Using Proposition $\ref{taylor prop}$, we have 
\begin{equation*}\|e^{t\mathcal{L}^g_\gamma} f\|_{\dot{H}^{\gamma, p}}=\|(-\triangle)^{\gamma/2}\frac{\mathcal{L}^g_\gamma}{\mathcal{L}^g_\gamma} e^{t\mathcal{L}^g_\gamma} f\|_{L^p}\leq \|\mathcal{L}^g_\gamma e^{t\mathcal{L}^g_\gamma} f\|_{L^p}\leq t^{-1}\|f\|_{L^p},
\end{equation*}
provided the operator $P=\frac{(-\triangle)^{\gamma/2}}{\mathcal{L}_\gamma}$ is bounded from $L^p(\mathbb{R}^n)$ to itself.  Since the symbol for $P$ is given by $\frac{|\xi|^{\gamma}}{|\xi|^\gamma/g(|\xi|)}=g(|\xi|)$, this follows directly from the assumption on $g$.  Standard interpolation and duality arguments extend this result to the general case of $s_1\leq s_2$.

We again state the parallel result for the special case where $g$ is, essentially, a logaritm.
\begin{lemma}\label{sobolev log}Let $1<p<\infty$, $s_1\leq s_2$, let $g(r)\leq Cr^{\varepsilon}$ for any $\varepsilon>0$ and let $|g^{(k)}(r)|\leq C|r|^{-k}$ for all $1\leq k\leq n/2+1$.   Then $e^{t\mathcal{L}_\gamma}: H^{s_1,p}(\mathbb{R}^n)\rightarrow H^{s_2,p}(\mathbb{R}^n)$ and 
\begin{equation}\|e^{t\mathcal{L}_\gamma} f\|_{H^{s_2,p}}\leq t^{-(s_2-s_1)/(\gamma-\varepsilon)}\|f\|_{H^{s_1,p}},
\end{equation}
for any $\varepsilon>0$.
\end{lemma}
The proof of this lemma is almost identical to the proof of Proposition $\ref{sobolev}$.  With the notation used at the beginning of that argument, we have  
\begin{equation*}\|e^{t\mathcal{L}_\gamma} f\|_{\dot{H}^{\gamma-\varepsilon, p}}=\|(-\triangle)^{(\gamma+\varepsilon)/2}\frac{\mathcal{L}_\gamma}{\mathcal{L}_\gamma} e^{t\mathcal{L}_\gamma} f\|_{L^p}\leq \|\mathcal{L}_\gamma e^{t\mathcal{L}_\gamma} f\|_{L^p}\leq t^{-1}\|f\|_{L^p},
\end{equation*}
where we again need the operator $P=\frac{(-\triangle)^{\gamma/2}}{\mathcal{L}_\gamma}$ to be bounded from  $L^p(\mathbb{R}^n)$ to itself.  In this case, the symbol for $P$ is given 
\begin{equation*}P\frac{|\xi|^{\gamma-\varepsilon}}{|\xi|^\gamma/g(|\xi|)}=|\xi|^{-\varepsilon} g(|\xi|),
\end{equation*}
and a straightforward calculation shows that the assumptions on $g$ make this symbol a Mihklin multiplier, and so the operator is $L^p(\mathbb{R}^n)$ bounded.  Standard interpolation and duality arguments extend this result to the general case of $s_1\leq s_2$.

The following is direct combination of Proposition $\ref{lpsemi}$ and Proposition $\ref{sobolev}$.

\begin{proposition}\label{heat kernel bound} Let $1<p<\infty$, $s_1\leq s_2$, $p_1\leq p_2$ and let $g$ satisfy the assumptions required by Proposition $\ref{lpsemi}$ and Proposition $\ref{sobolev}$.  Then $e^{t\mathcal{L}_\gamma}: H^{s_1,p_1}(\mathbb{R}^n)\rightarrow H^{s_2,p_2}(\mathbb{R}^n)$ and 
\begin{equation}\|e^{t\mathcal{L}_\gamma} f\|_{H^{s_2,p_2}}\leq t^{-(s_2-s_1+n/p_1-n/p_2)/\gamma}\|f\|_{H^{s_1,p_1}}.
\end{equation}
\end{proposition}

We also record the analogous result for our special case.
\begin{lemma}\label{heat kernel bound log}Let $1<p<\infty$, $s_1\leq s_2$, $p_1\leq p_2$, $g(r)\leq Cr^{\varepsilon}$ for any $\varepsilon>0$, and let $|g^{(k)}(r)|\leq C|r|^{-k}$ for all $1\leq k\leq n/2+1$. .  Then $e^{t\mathcal{L}_\gamma}: H^{s_1,p_1}(\mathbb{R}^n)\rightarrow H^{s_2,p_2}(\mathbb{R}^n)$ and 
\begin{equation}\|e^{t\mathcal{L}_\gamma} f\|_{H^{s_2,p_2}}\leq t^{-(s_2-s_1+n/p_1-n/p_2)/(\gamma-\varepsilon)}\|f\|_{H^{s_1,p_1}},
\end{equation}
for any small $\varepsilon>0$.
\end{lemma}

We remark that these results can be easily extended to Besov spaces (see Section 2 in \cite{besovpaper2} for an example of a similar process applied to the standard Heat kernel).  We also remark that the content of Lemma $\ref{heat kernel bound log}$ is that if $g_1$ satisfies the assumptions of Lemma $\ref{heat kernel bound log}$, then the operator $e^{t\mathcal{L}_\gamma^g}$ satisfies the same operator bounds as $e^{t\mathcal{L}_{\gamma-\varepsilon}}$ provided $\gamma-\varepsilon>1$.

\section{Higher regularity for the local existence result}\label{Higher regularity for the local existence result}
As was mentioned in the introduction, the solutions to the generalized Leray-alpha equations constructed here are smooth for all $t>0$ (provided the solution exists at time $t$).  In this section we prove that the solutions to Theorem $\ref{old style short}$ have this additional regularity and quantify the blow-up that occurs in these higher regularity norms as $t\rightarrow 0$.  We use an induction argument inspired by the results in \cite{katoinduction} for the Navier-Stokes equation.  We remark that similar results can be proven for the other theorems in this paper, but require slightly different (and in some cases much more involved) arguments.

\begin{proposition}\label{higher regularity theorem}Let $u_0\in B^{s_1}_{p,q}(\mathbb{R}^n)$ be divergence-free.   Let $u$ be a solution to the generalized Leray-alpha equation ($\ref{leray1})$  given by Theorem $\ref{old style short}$.
Then for all $r\geq s_1$ we have that $u\in \dot{C}^T_{(r-s_1)/2;r,p,q}$.
\end{proposition}

Before starting the proof, recall from Theorem $\ref{old style short}$ that $s_1>0$ and that 
\begin{equation}\label{durant}\aligned \gamma_1&>1
\\ \gamma_2&>0
\\ s_2&>\gamma_2
\\ s_2-s_1&<\min\{\gamma_1/2, 1\}
\\ \gamma_1&\geq s_2-s_1+n/p+1
\endaligned
\end{equation}

\begin{proof}
We start with the solution $u$ given by Theorem $(\ref{old style short}$.  Then let $\delta>0$ be arbitrary and let $v=t^\delta u$.  We note that $v(0)=0$.  Then 
\begin{equation*}\aligned \partial_t v&=\delta t^{\delta-1} u+t^\delta \partial_t u
\\ &=\delta t^{-1} v+t^\delta (\lone u-(1-\alpha^2\ltwo)^{-1}\div(u\tensor (1-\alpha^2\ltwo)u))
\\ &=\delta t^{-1} v+ \lone v-t^{-\delta}(1-\alpha^2\ltwo)^{-1}\div(u\tensor (1-\alpha^2\ltwo)u)).
\endaligned
\end{equation*}
Applying Duhamel's principle, we get 
\begin{equation*}\aligned v&=e^{t\lone}v_0+\int_0^t e^{(t-s)\lone}s^{-1}v(s)ds+\int_0^t e^{(t-s)\lone}s^{-\delta}W^\alpha(v(s),v(s))ds
\\ &=\int_0^t e^{(t-s)\lone}v(s)ds+\int_0^t e^{(t-s)\lone}W^\alpha(v(s),v(s))ds,
\endaligned
\end{equation*}
where we recall $W^\alpha(f,g)=(1-\mathcal{L}_2)^{-1}\div(f(s)\tensor (1+\mathcal{L}_2)g(s))$ and in the last line used that $v_0=0$.  Using $v=t^\delta u$, we get 
\begin{equation*}u=t^{-\delta}\int_0^t e^{(t-s)\lone}s^{\delta-1}u(s)ds+t^{-\delta}\int_0^t e^{(t-s)\lone}s^\delta W^\alpha(u(s),u(s))ds.
\end{equation*}

The key idea here is that we can choose $\delta$ to be large enough to cancel any singularities that occur at $s=0$.  Now we are ready to set up the induction.  We have by Theorem $\ref{old style short}$ that the local solution $u$ is in $\dot{C}^T_{(s_2-s_1)/2;s_2,p,q}$, where $s_2>\gamma_2/2$ .  For induction, we assume this solution $u$ is also in $\dot{C}^T_{(k-s_1)/2;k,p,q}$ for some $k\geq s_2$, and seek to show that $u$ is in $\dot{C}^T_{a_1;k+h,\bar{p}}$, where $a_1=(k+h-s_1)/2$ and $h$ is a fixed number between $0$ and $1$ which will be chosen later.  We have  
\begin{equation}\label{hugo}\aligned &\|u\|_{B^{k+h}_{p,q}}=t^{-\delta}\(\int_0^t \|e^{(t-s)\lone} s^{\delta-1}u(s)\|_{B^{k+h}_{p,q}}ds + \int_0^t \|e^{(t-s)\lone} s^{\delta}W^\alpha(u(s))\|_{B^{k+h}_{p,q}}ds\)
\\ \leq &Ct^{-\delta}\int_0^t |t-s|^{-h/\gamma_1}s^{\delta-1}\|u(s)\|_{B^{k}_{p,q}} + t^{-\delta}\int_0^t |t-s|^{-b_1/\gamma_1}s^{\delta}\|W^\alpha(u)\|_{B^{k-1}_{\tilde{p},q}} ds,
\endaligned
\end{equation}
where $b_1=h+1+n/p-n/\tilde{p}$.

For the first term in the right hand side of $(\ref{hugo})$, we have 
\begin{equation}\label{hugo2}\aligned &t^{-\delta}\int_0^t |t-s|^{-h/\gamma_1}s^{\delta-1}\|u(s)\|_{B^{k}_{p,q}}
\\ =&t^{-\delta}\|u\|_{(k-s_1)/\gamma_1;k,p,q}\int_0^t |t-s|^{-h/\gamma_1}s^{\delta-1-(k-s_1)/\gamma_1}ds
\\ \leq &C\|u\|_{(k-s_1)/\gamma_1;k,p,q}t^{-\delta}t^{-h/\gamma_1}t^{\delta-1-(k-s_1)/\gamma_1+1}
\\ \leq &Ct^{-(k+h-s_1)/\gamma_1}\|u\|_{(k-s_1)/\gamma_1;k,p,q}
\endaligned
\end{equation}
This calculation implicitly assumes that the exponents of $|t-s|$ and $s$ in the integral are both strictly greater than negative $1$.  For $|t-s|$, this holds provided $h/\gamma_1<1$.  For $s$, it works for a sufficiently large choice of $\delta$.  We note that without modifying the PDE to include these $t^\delta$ terms, we would need $(k-s_1)/\gamma_1$ to be less than $1$, which does not hold for large $k$.

For the second piece, we start by bounding $\|W^\alpha(u)\|_{B^{k-1}_{\tilde{p},q}}$.  Using Proposition $\ref{product est 1}$, we have 
\begin{equation}\label{mal38}\aligned &\|W(u(s),u(s))\|_{B^{k-1}_{p,q}}\leq \|u\tensor (1+\mathcal{L}_2)u\|_{B^{k-\gamma_2}_{p,q}}
\\ \leq &\|u\|_{L^{p_1}}\|(1+\mathcal{L}_2)u\|_{B^{k-\gamma_2}_{p_2,q}}+\|u\|_{B^{k-\gamma_2}_{q_1, q}}\|(1+\mathcal{L}_2)u\|_{L^{q_2}}
\\ \leq &\|u\|_{B^{r_1}_{p,q}}\|u\|_{B^{k+h}_{p,q}}+\|u\|_{B^{r_2}_{p,q}}\|u\|_{B^{r_3}_{p,q}},
\endaligned 
\end{equation}
provided 
\begin{equation}\label{aldi}\aligned 1/\tilde{p}&=1/p_1+1/p_2=1/q_1+1/q_2,
\\ n/p&>r_1, h, r_2, r_3,
\\ p_1&=\frac{np}{n-r_1p},
\\ p_2&=\frac{np}{n-hp},
\\ q_1&=\frac{np}{n-r_2p},
\\ q_1&=\frac{np}{n-r_3p}.
\endaligned
\end{equation}

Using $(\ref{mal38})$ in the last term in $(\ref{hugo})$, and setting $a_2=(r_1-s_1)/\gamma_1$, $(k+h-s_1)/\gamma_1$ we have 
\begin{equation}\label{hugo3}\aligned &t^{-\delta}\int_0^t |t-s|^{-(h+1)/\gamma_1}s^{\delta}\|W^\alpha(u)\|_{H^{k-1}_{\tilde{p},q}} ds 
\\ \leq &Ct^{-\delta}\|u\|_{a_2;p, q}\|u\|_{a_1;p, q} \int_0^t |t-s|^{-(h+1+n/p-r_1-h)/\gamma_1}s^{\delta-(r_1-s_1)/\gamma_1-(k+h-s_1)/\gamma_1}ds 
\\+ &Ct^{-\delta}\|u\|^2_{(r_2-s_1)/\gamma_1;p, q} \int_0^t |t-s|^{-(h+1+n/p-r_2-r_3)/\gamma_1}s^{\delta-(k-\gamma_2+r_2-s_1)/\gamma_1-(\gamma_2+r_3-s_1)/\gamma_1}ds
\\ \leq &C\|u\|_{a_2;p, q}\|u\|_{a_1;p, q}t^{-(n/p-r_1+1)/\gamma_1-(r_1-s_1)/\gamma_1-(k+h-s_1)/\gamma_1+1} 
\\ + &C\|u\|^2_{(r_2-s_1)/\gamma_1;p, q}t^{-(h+1+n/p-r_2-r_3)/\gamma_1-(k-\gamma_2+r_2-s_1)/\gamma_1-(\gamma_2+r_3-s_1)/\gamma_1+1}
\\ \leq &C\|u\|_{a_2;p, q}\|u\|_{a_1;p, q}t^{-(n/p-s_1+1)/\gamma_1-(k+h-s_1)/\gamma_1+1} 
\\ + &C\|u\|^2_{(r_2-s_1)/\gamma_1;p, q} t^{-(1+n/p-s_1)/\gamma_1-(k+h-s_1)/\gamma_1+1}
\\ \leq &C(|u\|_{a_2;p, q}\|u\|_{a_1;p, q}+ \|u\|^2_{(r_2-s_1)/\gamma_1;p, q}) t^{-(k+h-s_1)/\gamma_1+s_2/\gamma_1},
\endaligned
\end{equation}
where the last line used, from $(\ref{durant})$, that $s_2-s_1+n/p+1<\gamma_1$ (and thus $1-(n/p-s_1+1)/\gamma_1>s_2/\gamma_1$).

Using inequalities $(\ref{hugo2})$ and $\ref{hugo3}$ in $(\ref{hugo})$, we have 
\begin{equation*}\|u\|_{B^{k+h}_{p,q}}\leq C\(\|u\|_{(k-s_1)/\gamma_1;k,p,q}+|u\|_{a_2;p, q}\|u\|_{a_1;p, q}+ \|u\|^2_{(r_2-s_1)/\gamma_1;p, q}\)t^{-(k+h-s_1)/\gamma_1+s_2/\gamma_1}.
\end{equation*}

Multiplying both sides by $t^(k+h-s_1)/\gamma_1$ completes the argument.  We remark that $\delta$ is chosen after beginning the induction step, while the appropriate value of $h$ is fixed by the parameters $n, p, s_1, s_2$.
\end{proof}

\newpage
\bibliographystyle{amsplain}
\bibliography{references}

\providecommand{\bysame}{\leavevmode\hbox to3em{\hrulefill}\thinspace}
\providecommand{\MR}{\relax\ifhmode\unskip\space\fi MR }
\providecommand{\MRhref}[2]{%
  \href{http://www.ams.org/mathscinet-getitem?mr=#1}{#2}
}
\providecommand{\href}[2]{#2}
\begin{thebibliography}{10}

\bibitem{chae}
D.~Chae, \emph{Local existence and blow-up criterion for the {E}uler equations
  in the {B}esov spaces}, Asymptotic {A}nalysis \textbf{38} (2004), 339--358.

\bibitem{chemin}
J.Y. Chemin, \emph{About the navier-stokes system}, Publications du Laboratoire
  d'analyse numerique (1996).

\bibitem{galplan}
I.~Gallagher and F.~Planchon, \emph{On global infinite energy solutions to the
  {N}avier-{S}tokes equations in two dimensions}, Arch. Ration. Mech. Anal.
  \textbf{161} (2002), no.~4, 307--337.

\bibitem{holmleray}
J.~D. Gibbon and D.~D. Holm, \emph{Estimates for the {LANS}-{$\alpha$},
  {L}eray-{$\alpha$} and {B}ardina models in terms of a {N}avier-{S}tokes
  {R}eynolds number}, Indiana Univ. Math. J. \textbf{57} (2008), no.~6,
  2761--2773.

\bibitem{Kato}
T.~Kato, \emph{Strong ${L}^p$-solutions of the {N}avier {S}tokes {E}quation in
  $\mathbb{R}^m$, with {A}pplications to {W}eak {S}olutions}, Mathematische
  {Z}eitschrift \textbf{187} (1984), 471--480.

\bibitem{katoinduction}
\bysame, \emph{The {N}avier-{S}tokes equation for an incompressible fluid in
  {${\bf R}^2$} with a measure as the initial vorticity}, Differential Integral
  Equations \textbf{7} (1994), no.~3-4, 949--966.

\bibitem{KP}
T.~Kato and G.~Ponce, \emph{The {N}avier-{S}tokes equation with weak initial
  data}, Internat. Math. Res. Notices (1994), no.~10.

\bibitem{KT}
H.~Koch and D.~Tataru, \emph{Well-posedness for the {N}avier-{S}tokes
  equations}, Adv. Math. \textbf{157} (2001), no.~1, 22--35.

\bibitem{lady356}
O.~A. Ladyzhenskaya, \emph{The mathematical theory of viscous incompressible
  flow}, Second English edition, revised and enlarged. Translated from the
  Russian by Richard A. Silverman and John Chu. Mathematics and its
  Applications, Vol. 2, Gordon and Breach Science Publishers, New York, 1969.

\bibitem{nsbook}
P.~G. Lemarie-Rieusset, \emph{Recent developments in the navier-stokes
  problem}, Chapman and Hall/CRC, 2002.

\bibitem{olsontitialphalike}
Eric Olson and Edriss~S. Titi, \emph{Viscosity versus vorticity stretching:
  global well-posedness for a family of {N}avier--{S}tokes-alpha-like models},
  Nonlinear Analysis. Theory, Methods \& Applications. An International
  Multidisciplinary Journal. Series A: Theory and Methods \textbf{66} (2007),
  no.~11, 2427--2458.

\bibitem{sobpaper}
N.~Pennington, \emph{Lagrangian {A}veraged {N}avier-{S}tokes equation with
  rough data in {S}obolev spaces}, Journal of Mathematical Analysis and
  Applications \textbf{403}.

\bibitem{besovpaper2}
\bysame, \emph{Local and global existence for the {L}agrangian {A}veraged
  {N}avier-{S}tokes equation in {B}esov spaces}, Electron. J. Diff. Equ.
  \textbf{2012} (2012), no.~89, 1--19.

\bibitem{taolog}
Terence Tao, \emph{Global regularity for a logarithmically supercritical
  hyperdissipative {N}avier-{S}tokes equation}, Analysis \& PDE \textbf{2}
  (2009), no.~3, 361--366.

\bibitem{T3}
M.~Taylor, \emph{Partial {D}ifferential {E}quations}, Springer-{V}erlag {N}ew
  {Y}ork, {I}nc., 1996.

\bibitem{TT}
\bysame, \emph{Tools for {P}{D}{E}}, {M}athematical {S}urveys and {M}onographs,
  {V}ol. 81, {A}merican {M}athematical {S}ociety, Providence {R}{I}, 2000.

\bibitem{Wu}
J.~Wu, \emph{Lower {B}ounds for an {I}ntegral {I}nvolving {F}ractional
  {L}aplacians and {G}eneralized {N}avier-{S}tokes {E}quations in {B}esov
  {S}paces}, Communications in {M}athematical {P}hysics \textbf{263} (2005),
  803--831.

\bibitem{kazuo}
Kazuo Yamazaki, \emph{On the global regularity of generalized {L}eray-alpha
  type models}, Nonlinear Analysis. Theory, Methods \& Applications. An
  International Multidisciplinary Journal. Series A: Theory and Methods
  \textbf{75} (2012), no.~2, 503--515.

\bibitem{generalizedmhd}
Zhuan Ye and Xiaojing Xu, \emph{Global regularity of the two-dimensional
  incompressible generalized magnetohydrodynamics system}, Nonlinear Analysis.
  Theory, Methods \& Applications. An International Multidisciplinary Journal.
  Series A: Theory and Methods \textbf{100} (2014), no.~0, 86--96.

\end{thebibliography}

\end{document}